\documentclass[12pt]{article}

\usepackage[T1]{fontenc}
\usepackage[utf8]{inputenc}
\usepackage{pifont}
\usepackage{mathrsfs}
\usepackage{graphicx,ifpdf}
\usepackage{amsmath,amsfonts,amsthm,amssymb}
\usepackage{bbding}
\usepackage{graphics,multicol}
\usepackage{color}
\usepackage{fourier}
\usepackage{float}
\usepackage{enumerate}
\usepackage{enumitem}
\usepackage[hypertexnames=false,hidelinks]{hyperref}
\usepackage[numbers]{natbib}
\usepackage{thm-restate}
\usepackage{indentfirst}

\allowdisplaybreaks[4]
\evensidemargin=\oddsidemargin
\newtheorem{conj}{Conjecture}
\newtheorem{lem}{Lemma}[section]

\theoremstyle{remark}

\theoremstyle{plain}

\begin{document}
	
	\title{Decomposing multigraphs into two Class I subgraphs with prescribed maximum degrees\thanks{This work is supported by National Natural Science Foundation of China (Grant No. 12301433).}}
	
	\author{{\small Yonglei Chen,\ \ Yan Cao \footnote{Corresponding author. Email: ycao@dlut.edu.cn (Yan Cao);  ylchen981127@mail.dlut.edu.cn (Yonglei Chen).}}\\[2mm]
		{\footnotesize School of Mathematical Sciences, Dalian University of Technology, Dalian, Liaoning 116024, China}}
	
	\date{}
	\maketitle
	
	{\flushleft\large\bfseries Abstract.}
	Let $G$ be a finite loopless multigraph with maximum degree $\Delta$ and
	$\chi'(G)=\Delta+k$, where $k\ge0$. We prove that for every pair of positive
	integers $p,q$ satisfying
	$$
	p+q=\Delta+k
	\qquad\text{and}\qquad
	p,q\le\Delta,
	$$
	the graph $G$ admits an edge decomposition into two spanning Class I subgraphs
	$H_p,H_q$ with
	$$
	\Delta(H_p)=p
	\qquad\text{and}\qquad
	\Delta(H_q)=q.
	$$
	
	{\flushleft\large\bf Keywords:} Maximum $t$-edge-colorable subgraph, Chromatic index, Class I graph \vspace{1mm}
	
	\section{Introduction}\label{sec:introduction}
	
	Throughout this paper, graphs are finite and loopless, and parallel edges are
	allowed. For a graph $G$, let $V(G)$ and $E(G)$ denote its vertex set and edge set, respectively. For $v\in V(G)$, let $E_G(v)$ denote the set of edges incident with $v$, and let $d_G(v)=|E_G(v)|$.
	We write $\Delta(G)$ and $\chi'(G)$ for the maximum
	degree and chromatic index of $G$, respectively. We call $G$ \emph{Class I} if
	$\chi'(G)=\Delta(G)$, and \emph{Class II} otherwise. For a positive integer $t$, a $t$-edge-colorable subgraph $F$ of $G$ is called a \emph{maximum $t$-edge-colorable subgraph of $G$}
	if it has as many edges as possible among all $t$-edge-colorable subgraphs of
	$G$.
	
	Maximum edge-colorable subgraphs have been studied from extremal,
	structural, and algorithmic perspectives in several settings
	\cite{AlbertsonHaas1996,AslanyanEtAl2014,KaminskiKowalik2014,
		KarapetyanMkrtchyan2019,Kosowski2009,MkrtchyanPetrosyanVardanyan2010,
		Rizzi2009,Steffen2004}. 
	When a graph is not $t$-edge-colorable, a maximum
	$t$-edge-colorable subgraph describes the largest portion of its edge set
	that can be properly colored with $t$ colors. In the case
	$t=\Delta(G)$, Mkrtchyan and Steffen~\cite{MkrtchyanSteffen2012} studied
	maximum $\Delta(G)$-edge-colorable subgraphs of Class II graphs. They pointed
	out that there are long-standing open problems concerning Class II graphs
	and that even obtaining structural information on such graphs is difficult.
	Their work investigated both the size and structural properties of maximum
	$\Delta(G)$-edge-colorable subgraphs.
	
	In particular, Mkrtchyan and Steffen~\cite{MkrtchyanSteffen2012}  proved that
	every simple Class II graph $G$ has a maximum
	$\Delta(G)$-edge-colorable subgraph $H$ such that
	$E(G)\setminus E(H)$ is a matching. They also showed that every maximum
	$\Delta(G)$-edge-colorable subgraph of a simple graph is itself Class I.
	Thus every simple Class II graph can be decomposed into a maximum
	$\Delta(G)$-edge-colorable subgraph and a matching, both of which are
	Class I. They further conjectured that, whenever
	$\chi'(G)=\Delta(G)+k$, there exists a maximum
	$\Delta(G)$-edge-colorable subgraph whose complement is
	$k$-edge-colorable.
	
	Building on the work of Mkrtchyan and Steffen, Cao et
	al.~\cite{CaoEtAl2023} studied decompositions of Class II graphs into  Class I subgraphs
	with prescribed degree conditions. They first extended the
	above decomposition result from simple graphs to multigraphs. More precisely,
	they proved that if $G$ has multiplicity $\mu$, then $G$ can be decomposed
	into a maximum $\Delta(G)$-edge-colorable subgraph and a subgraph of maximum
	degree at most $\mu$. They further proved that if
	\[
	\chi'(G)=\Delta(G)+k,\qquad k\geq 1,
	\]
	then $G$ can be decomposed into two Class I subgraphs $H_1$ and $H_2$
	satisfying
	\[
	\Delta(H_1)=\Delta(G)
	\qquad\text{and}\qquad
	\Delta(H_2)=k.
	\]
	
	Cao et al.~\cite{CaoEtAl2023} further conjectured that the latter result could be generalized as follows.

\begin{conj}[Cao et al.\ \cite{CaoEtAl2023}]
\label{conj:cao}
Let $G$ be a graph with $\chi'(G)=\Delta(G)+k$,
and let $p,q$ be positive integers satisfying $p+q=\Delta(G)+k$ with
$p,q\le\Delta(G)$. Then $G$ can be decomposed into two Class I subgraphs
$H_1$ and $H_2$ such that $\Delta(H_1)=p$ and $\Delta(H_2)=q$.
\end{conj}

We prove Conjecture~\ref{conj:cao} in the following form.

\begin{restatable}[Main Theorem]{thm}{MainTheorem}
\label{thm:conjecture4.1}
Let $G$ be a finite loopless multigraph with maximum degree $\Delta$ and
$\chi'(G)=\Delta+k$, $k\ge0$. Let $p,q$ be positive integers satisfying
$p+q=\Delta+k$ and $p,q\le\Delta$. Then $G$ has an edge decomposition
$$
E(G)=E(H_p)\mathbin{\dot\cup}E(H_q)
$$
into two spanning Class I subgraphs such that
$$
\Delta(H_p)=p
\qquad\text{and}\qquad
\Delta(H_q)=q.
$$
\end{restatable}

~

Theorem~\ref{thm:conjecture4.1} strengthens the result of Cao et al. \cite{CaoEtAl2023}, which guarantees a decomposition into two Class I subgraphs with maximum degrees $\Delta(G)$ and $k$. Our result shows that the maximum degrees of the two subgraphs can be prescribed arbitrarily, subject to $p+q=\chi'(G)$ and $p,q\le\Delta(G)$. 

For the proofs, we use the following notation. For a positive integer $r$, let $[r]=\{1,2,\ldots,r\}$. Given a proper $r$-edge-coloring $\theta$ of $G$
and $v\in V(G)$, let $\theta(v)=\{\theta(e):e\in E_G(v)\}$ and $\bar{\theta}(v)=[r]\setminus\theta(v)$ denote the sets of colors present and missing at a vertex $v$, respectively. We say that $v$ sees a color $c\in[r]$ if $c\in\theta(v)$, and that $v$ misses $c$ if $c\in\bar{\theta}(v)$.

For distinct colors $\alpha,\beta\in[r]$, let $G_\theta(\alpha,\beta)$ denote the spanning subgraph of $G$ consisting of the edges colored $\alpha$ or $\beta$
under $\theta$. The components of this subgraph are
called $(\alpha,\beta)$-chains. For $v\in V(G)$, let $P_v(\alpha,\beta,\theta)$ denote the $(\alpha,\beta)$-chain containing $v$.
 Two vertices $x$ and $y$ are said to be
\emph{$(\alpha,\beta)$-linked under $\theta$} if they belong to the same component of $G_\theta(\alpha,\beta)$. Interchanging the colors $\alpha$ and $\beta$ on such a chain is called a Kempe change.

	\section{Lemmas}
	
	Throughout this section, put $\Delta=\Delta(G)$ and $r=\chi'(G)=\Delta+k$, where $k\ge 1$.
	Let $p,q$ be positive integers satisfying $p+q=r$ and $p, q\le\Delta.$
	Assume without loss of generality that $p\le q$. Then $q\ge p=r-q=\Delta+k-q\ge k$.
	
	Among all pairs $(\varphi,c)$, where $\varphi$ is a proper $r$-edge-coloring of $G$ and $c\in [r]$, choose one minimizing $$|\varphi^{-1}(c)|,$$ where $\varphi^{-1}(c)=\{e\in E(G):\varphi(e)=c\}$.
	By renaming colors, we may let $c=r$ and put
	$$E_r(\varphi)=\varphi^{-1}(r).$$
	Thus $E_r(\varphi)$ has minimum cardinality among all color classes of all
	proper $r$-edge-colorings of $G$. We write $E_r$ for $E_r(\varphi)$ whenever
	the coloring is clear from the context.
	
	\begin{lem}
		\label{lem:minclass-rigidity}
		The set $E_r$ is nonempty. More generally, let $\theta$ be any proper $r$-edge-coloring of $G$ such that $|E_r(\theta)|=|E_r|$, where $E_r(\theta)=\theta^{-1}(r)$, and let $uv\in E_r(\theta)$.
		Then the following hold:
		\begin{enumerate}[label=\textup{(\roman*)}]
			\item $\bar{\theta}(u)\cap \bar{\theta}(v)=\varnothing;$
			\item $d_G(u)+d_G(v)\ge r+1;$
			\item for every $\alpha\in \bar{\theta}(u)$ and $\beta\in \bar{\theta}(v)$, the vertices
			$u$ and $v$ are $(\alpha,\beta)$-linked under $\theta$.
		\end{enumerate}
	\end{lem}
	
	\begin{proof}
		Note that $E_r\ne\varnothing$ since $\chi'(G)=r$.
		
		For (i), suppose that $\alpha\in \bar{\theta}(u)\cap \bar{\theta}(v)$. Since $\theta(uv)=r$, we have $\alpha\ne r$. Recolor $uv$ with $\alpha$. If $|E_r(\theta)|\ge2$, the resulting proper $r$-edge-coloring has a
		color-$r$ class of size $|E_r(\theta)|-1$, contradicting the
		minimality of $|E_r|$. If $|E_r(\theta)|=1$, then we obtain a
		proper $(r-1)$-edge-coloring, contradicting $r=\chi'(G)$. Hence
		$\bar{\theta}(u)\cap \bar{\theta}(v)=\varnothing$.
		
		For (ii), $r\notin \bar{\theta}(u)\cup\bar{\theta}(v)$. By (i), $|\bar{\theta}(u)|+|\bar{\theta}(v)|\le r-1$.
		Since $|\bar{\theta}(u)|=r-d_G(u)$ and $|\bar{\theta}(v)|=r-d_G(v)$,
		we obtain $d_G(u)+d_G(v)\ge r+1$.
		
		For (iii), suppose that $u$ and $v$ are not $(\alpha,\beta)$-linked for some $\alpha\in \bar{\theta}(u)$ and $\beta\in \bar{\theta}(v)$.
		Interchange $\alpha$ and $\beta$ on the $(\alpha,\beta)$-component
		containing $u$. Let $\theta'$ be the resulting proper $r$-edge-coloring. Since $v$ lies in a different component, it is unchanged
		and still misses $\beta$, while $u$ now also misses $\beta$. Since $r\notin\{\alpha, \beta \}$, $E_r(\theta')=E_r(\theta)$. Moreover, $\beta\in \bar{\theta'}(u)\cap\bar{\theta'}(v)$, contrary to (i).
	\end{proof}

	\begin{lem}
		\label{lem:threshold}
		Let $\psi$ be a proper $r$-edge-coloring of $G$. Suppose vertices $x,y$
		satisfy $\bar{\psi}({x})\cap\bar{\psi}({y})=\varnothing$, $d_G(x)\ge p$ and $d_G(y)\ge q$.
		Then $G$ has an edge decomposition
		$$
		E(G)=E(H_p)\mathbin{\dot\cup}E(H_q)
		$$
		into two spanning Class I subgraphs satisfying
		$$
		\Delta(H_p)=p
		\qquad\text{and}\qquad
		\Delta(H_q)=q.
		$$
	\end{lem}
	
	\begin{proof}
		Let $X=\bar{\psi}(x)$ and $Y=\bar{\psi}(y)$.
		Since $X\cap Y=\varnothing$, $Y\subseteq[r]\setminus X$.	
		Moreover, $|Y|=r-d_G(y)\le r-q=p$
		and $|[r]\setminus X|=d_G(x)\ge p$.
		Hence there exists $S\subseteq[r]$ such that $Y\subseteq S\subseteq[r]\setminus X$ and $|S|=p$.
		Let $T=[r]\setminus S$.
		Then $|T|=q$.
		
		Let $H_p$ be the spanning subgraph formed by the color classes in $S$, and
		let $H_q$ be the spanning subgraph formed by the color classes in $T$.
		Then
		$$
		E(G)=E(H_p)\mathbin{\dot\cup}E(H_q).
		$$
		Since $H_p$ is the union of $p$ color classes, $\chi'(H_p)\le p$ and $\Delta(H_p)\le p$.
		Because $S\subseteq[r]\setminus X$, every color in $S$ appears at $x$; and so $d_{H_p}(x)=p$.
		Thus
		$$
		\chi'(H_p)=\Delta(H_p)=p.
		$$
		
		Similarly, $Y\subseteq S$ gives $T\cap Y=\varnothing$, so every color in
		$T$ appears at $y$. Hence $d_{H_q}(y)=q$, 
		and therefore
		$$
		\chi'(H_q)=\Delta(H_q)=q.
		$$
	\end{proof}
	
	For the remainder of this section, suppose, to the contrary, that
	Theorem~\ref{thm:conjecture4.1} is false. Fix $\varphi$ and the minimum
	color class $E_r$ as above.
	
	\begin{lem}
		\label{lem:middle}
		Let $uv\in E_r$, relabelled so that $d_G(u)\le d_G(v)$. Then
		$$
		p+2\le d_G(u)\le d_G(v)\le q-1.
		$$
		In particular,
		$$
		p\le d_G(u),d_G(v)\le q-1.
		$$
	\end{lem}
	
	\begin{proof}
		Since the theorem is assumed false, Lemma~\ref{lem:threshold} cannot apply
		to the endpoints of $uv$. By Lemma~\ref{lem:minclass-rigidity}(i), $\bar{\varphi}(u)\cap \bar{\varphi}(v)=\varnothing$. Hence it is impossible that both
		$$
		d_G(u)\ge p
		\qquad\text{and}\qquad
		d_G(v)\ge q.
		$$
		Therefore $d_G(u)\le p-1$ or $d_G(v)\le q-1$.
		
		By Lemma~\ref{lem:minclass-rigidity}(ii),
		$d_G(u)+d_G(v)\ge r+1=p+q+1.$
		Consequently, exactly one of the following two alternatives holds:
		\begin{enumerate}[label=\textup{(\roman*)}]
			\item $$d_G(u)\le p-1,
			\qquad
			d_G(v)\ge q+2;$$
			\item
			$$
			p+2\le d_G(u)\le d_G(v)\le q-1.
			$$
		\end{enumerate}
		It remains to exclude (i).
		
		Suppose, to the contrary, that $d_G(u)\le p-1$ and $d_G(v)\ge q+2$.
		Let $A=\bar{\varphi}(u)$ and $B=\bar{\varphi}(v)$.
		The set $B$ is nonempty because $d_G(v)\le\Delta<r$. Choose $\beta\in B$.
		Since $A\cap B=\varnothing$, $u$ sees $\beta$. Let
		$ux_\beta$ be the unique edge incident with $u$ that is colored $\beta$. Since $v$ misses
		$\beta$, $x_\beta\ne v$.
		
		Fix $\alpha\in A$. By Lemma~\ref{lem:minclass-rigidity}(iii), the vertices
		$u,v$ are $(\alpha,\beta)$-linked. Their common $(\alpha,\beta)$-component is a path joining $u$ and $v$ and contains $ux_\beta$. Since $x_\beta\ne v$, the vertex $x_\beta$ is internal on that path and hence is incident with an edge colored $\alpha$. This holds for every $\alpha\in A$. Therefore
		$$
		d_G(x_\beta)\ge |A|+1=r-d_G(u)+1\ge q+2.
		$$
		
		We claim $\bar{\varphi}(x_\beta)\cap B=\varnothing$.
		Suppose, to the contrary, that there exists $\gamma\in\bar{\varphi}(x_\beta)\cap B$.
		Since $x_\beta$ is incident with an edge colored $\beta$, $\gamma\ne\beta$.
		Choose $\alpha\in A$. By Lemma~\ref{lem:minclass-rigidity}(iii), the
		vertices $u,v$ lie in one $(\alpha,\gamma)$-component, which is a path $P$ joining $u$ and $v$. The vertex $x_\beta$ is incident with an edge colored $\alpha$ and misses $\gamma$, so it has degree one in $G_\varphi(\alpha,\gamma)$. Since $x_\beta\notin\{u,v\}$, it cannot be
		internal on $P$. Thus $x_\beta\notin V(P)$.
		Interchange $\alpha,\gamma$ on $P$. Then $\gamma$ is missing at $u$, while
		$x_\beta$ is unchanged and still misses $\gamma$. Recolor $ux_\beta$ by $\gamma$. 
		Now $\beta$ is missing at $u$, while $v$ still misses $\beta$. Recolor $uv$ by $\beta$.
		The result is a proper $r$-edge-coloring with one fewer edge of color $r$,
		contradicting the minimality of $|E_r|$. Therefore $\bar{\varphi}(x_\beta)\cap B=\varnothing$.
		
		Now
		$$
		d_G(x_\beta)\ge q+2\ge p,
		\qquad
		d_G(v)\ge q+2\ge q,
		$$
		and
		$$
		\bar{\varphi}(x_\beta)\cap\bar{\varphi}(v)=\varnothing.
		$$
		Lemma~\ref{lem:threshold} applies to $x_\beta,v$, a contradiction. Hence (i) is impossible, and (ii) follows.
	\end{proof}
	
	Fix $uv\in E_r(\varphi)$ and let $A=\bar{\varphi}(u)$ and $B=\bar{\varphi}(v)$.
	Choose $z$ with $d_G(z)=\Delta$. We have $|\bar{\varphi}(z)|=r-\Delta=k$.
	
	We first note that $r\in\bar{\varphi}(z)$.
	Otherwise $z$ is incident with an edge $zw\in E_r(\varphi)$. But Lemma~\ref{lem:middle}, applied to $zw$, gives $d_G(z)\le q-1$, contrary to $d_G(z)=\Delta\ge q$.
	
	\begin{lem}
		\label{lem:blocker-setup}
		Let $\theta$ be a proper $r$-edge-coloring of $G$ satisfying $E_r(\theta)=E_r(\varphi)$, $\bar{\theta}(u)=A$, $\bar{\theta}(v)=B$ and $r\in\bar{\theta}(z)$.
		Then
		$$
		|B\setminus\bar{\theta}(z)|
		\ge
		|A\cap\bar{\theta}(z)|+2.
		$$
		In particular, if $A\cap\bar{\theta}(z)\ne\varnothing$, then $B\setminus\bar{\theta}(z)\ne\varnothing$.
	\end{lem}
	
	\begin{proof}
		Since $d_G(z)=\Delta$, $|\bar{\theta}(z)|=r-\Delta=k$.
		Moreover, $uv\in E_r(\theta)$, and hence $r\notin A\cup B$.
		By Lemma~\ref{lem:minclass-rigidity}(i), $A\cap B=\varnothing$.
		Since $r\in\bar{\theta}(z)$, it follows that
		$$
		|A\cap\bar{\theta}(z)|
		+
		|B\cap\bar{\theta}(z)|
		\le k-1.
		$$
		
		On the other hand, Lemma~\ref{lem:middle} gives $d_G(v)\le q-1$.
		Therefore $|B|=r-d_G(v)\ge r-(q-1)=p+1\ge k+1$,
		where the last inequality follows from $p\ge k$. Consequently,
		$$
		\begin{aligned}
			|B\setminus\bar{\theta}(z)|
			&=
			|B|-|B\cap\bar{\theta}(z)|\\
			&\ge
			k+1-|B\cap\bar{\theta}(z)|\\
			&\ge
			|A\cap\bar{\theta}(z)|+2.
		\end{aligned}
		$$
		
	\end{proof}
	
	\begin{lem}
		\label{lem:blocker-elimination}
		There exists a proper $r$-edge-coloring $\psi$ satisfying $
		E_r(\psi)=E_r(\varphi)$, $\bar{\psi}(u)=A$, $\bar{\psi}(v)=B$, $r\in\bar{\psi}(z)$ and $A\cap\bar{\psi}(z)=\varnothing$.
	\end{lem}
	
	\begin{proof}
		Call a proper $r$-edge-coloring $\theta$ \emph{admissible}
		if $E_r(\theta)=E_r(\varphi)$, $\bar{\theta}(u)=A$, $\bar{\theta}(v)=B$ and $r\in\bar{\theta}(z)$.
		By the definitions of $A$ and $B$ and by the observation above, the coloring $\varphi$ is admissible.
		
		Among all admissible colorings, choose $\theta$ such that $|A\cap\bar\theta(z)|$ is minimum. Suppose, to the contrary, that $A\cap\bar{\theta}(z)\ne\varnothing$. Choose $a\in A\cap\bar{\theta}(z)$.
		By Lemma~\ref{lem:blocker-setup}, there exists $b\in B\setminus\bar{\theta}(z)$.
		
		Since $ \bar{\theta}(u)=A$, $\bar{\theta}(v)=B$ and $A\cap B=\varnothing$, the color $a$ is missing at $u$ and present at $v$, while the color $b$ is missing at $v$ and present at $u$. Moreover, $E_r(\theta)=E_r(\varphi)$ and $uv\in E_r(\theta).$ By Lemma~\ref{lem:minclass-rigidity}(iii), the vertices $u$ and $v$ are $(a,b)$-linked under $\theta$. Hence the $(a,b)$-component containing $u$ and $v$ is a path with endvertices $u$ and $v$.
		
		At $z$, the color $a$ is missing and the color $b$ is present. Thus $z$
		has degree one in $G_\theta(a, b)$. Moreover, $z\notin\{u,v\}$, because $d_G(z)=\Delta$, whereas Lemma~\ref{lem:middle} gives $d_G(u),d_G(v)\le q-1\le\Delta-1$.
		Hence $z$ cannot lie in the $(a,b)$-component containing $u$ and $v$: every non-endvertex of
		that path has degree two in $G_\theta(a, b)$, while $z$ has degree one.
		Thus $P_z(a, b, \theta)$ is distinct from the $(a,b)$-component containing $u$ and $v$.
		
		Interchange $a$ and $b$ on $P_z(a, b, \theta)$, obtaining a proper $r$-edge-coloring $\theta'$. Since $u,v\notin V(P_z(a, b, \theta))$, $\bar{\theta'}(u)=A$ and $\bar{\theta'}(v)=B$.
		Also $a,b\ne r$, because color $r$ is present at both $u$ and $v$.
		Therefore
		$$
		E_r(\theta')=E_r(\theta)=E_r(\varphi).
		$$
		Moreover,
		$$
		\bar{\theta'}({z})
		=
		(\bar{\theta}({z})\setminus\{a\})\cup\{b\}.
		$$
		Since $a\in A$ and $b\in B$ with $A\cap B=\varnothing$,
		$$ |A\cap\bar{\theta'}(z)| = |A\cap\bar{\theta}(z)|-1. $$
		Finally, $a, b\ne r$, so $r$ remains missing at $z$. Hence $\theta'$ is again admissible.
		
		This contradicts the choice of $\theta$. Therefore $A\cap\bar\theta(z)=\varnothing$.
		Taking $\psi=\theta$ completes the proof.
	\end{proof}
	
	\section{Proof of Theorem~\ref{thm:conjecture4.1}}
	
	\MainTheorem*
	
	\begin{proof}
		If $k=0$, then $\chi'(G)=\Delta$. Let $\varphi$ be a proper
		$\Delta$-edge-coloring of $G$, and choose a vertex $z$ with
		$d_G(z)=\Delta$. Partition the color set $[\Delta]$ into two sets
		$S$ and $T$ with $|S|=p$ and $|T|=q$. Let $H_p$ and $H_q$ be the
		spanning subgraphs formed by the color classes in $S$ and $T$,
		respectively. Since every color in $[\Delta]$ appears at $z$,
		we have
		\[
		d_{H_p}(z)=p
		\qquad\text{and}\qquad
		d_{H_q}(z)=q.
		\]
		Hence
		\[
		\chi'(H_p)=\Delta(H_p)=p,
		\qquad
		\chi'(H_q)=\Delta(H_q)=q.
		\]
		Thus the result holds when $k=0$. Henceforth assume $k\ge1$.
		
		Suppose, to the contrary, that the theorem fails. Choose $\varphi$ and the minimum color class $E_r(\varphi)$ as above, and fix an edge $uv\in E_r(\varphi)$.
		Relabel $u$ and $v$, if necessary, so that $d_G(u)\le d_G(v).$
		By Lemma~\ref{lem:middle}, $p+2\le d_G(u)\le d_G(v)\le q-1$.
		
		Let $A=\bar{\varphi}(u)$, $B=\bar{\varphi}(v)$, and choose a vertex $z$ with $d_G(z)=\Delta$.
		
		By Lemma~\ref{lem:blocker-elimination}, there exists a proper
		$r$-edge-coloring $\psi$ such that
		$$
		\bar{\psi}(u)=A
		$$
		and
		$$
		A\cap\bar{\psi}(z)=\varnothing.
		$$
		Hence
		$$
		\bar{\psi}(u)\cap\bar{\psi}(z)
		=
		A\cap\bar{\psi}(z)
		=
		\varnothing.
		$$
		
		Moreover,
		$$
		d_G(u)\ge p
		$$
		and
		$$
		d_G(z)=\Delta\ge q.
		$$
		Therefore Lemma~\ref{lem:threshold}, applied to $u$ and $z$ under the
		coloring $\psi$, yields an edge decomposition
		$$
		E(G)=E(H_p)\mathbin{\dot\cup}E(H_q)
		$$
		into two spanning Class I subgraphs satisfying
		$$
		\Delta(H_p)=p
		\qquad\text{and}\qquad
		\Delta(H_q)=q.
		$$
		This contradicts the assumption that $G$ is a counterexample.
	\end{proof}
	
	%	\section*{Declaration on the use of AI}
	%	
	%	The authors used ChatGPT solely for grammar checking, language polishing, and improving the clarity of the exposition. The authors used ChatGPT for language polishing and for checking the clarity and internalconsistency of the exposition. All mathematical ideas, proofs, and the overall structure of the paper were developed entirely by the authors. The authors take full responsibility for the correctness and originality of the paper.
	
	\bibliography{references_v9}
\end{document}